\documentclass{amsart}

\usepackage[T1]{fontenc}      
\usepackage[french]{babel}
\usepackage[utf8]{inputenc}
\usepackage{blindtext}

\usepackage{amsfonts,amscd,array, mathdots, epigraph}
\usepackage{amsmath}
\usepackage{amssymb}
\usepackage{amsthm}
\usepackage{mathrsfs}

\newtheorem{thm}{Th\'eor\`eme}[section]
\newtheorem{pro}[thm]{Proposition}

\theoremstyle{definition}
\newtheorem{rem}[thm]{Remarque}
\newtheorem{defn}[thm]{D\'efinition}
\newtheorem{exam}[thm]{Exemples}

\numberwithin{equation}{section}

\newcommand{\bZ}{\mathbb{Z}}
\newcommand{\bC}{\mathbb{C}}
\newcommand{\bN}{\mathbb{N^{*}}}
\newcommand{\bM}{\mathbb{N}}
\newcommand{\bR}{\mathbb{R}}

\begin{document}

\title{$\lambda$-quiddité sur $\bZ[\alpha]$ avec $\alpha$ transcendant}

\author{Flavien Mabilat}
\address{Laboratoire de Math\'ematiques de Reims,
UMR9008 CNRS et Université de Reims Champagne-Ardenne,
U.F.R. Sciences Exactes et Naturelles \\
Moulin de la Housse - BP 1039\\ 
51687 Reims cedex 2,\\
France}
\email{flavien.mabilat@univ-reims.fr}

\begin{abstract}
Dans le cadre de l'\'etude des frises de Coxeter, M.Cuntz a introduit la notion de $\lambda$-quiddit\'e irr\'eductible. L'objectif de cette note est de lister toutes les $\lambda$-quiddit\'es irr\'eductibles sur l'anneau $\bZ[\alpha]$ dans le cas o\`u $\alpha$ est un nombre complexe transcendant.
\end{abstract}

\maketitle

\begin{flushright}
 \textit{``Toute connaissance est une r\'eponse \`a une question.''} 
\\ Gaston Bachelard, \textit{La formation de l'esprit scientifique}
\end{flushright}

\section{Introduction}
\label{Intro}

Les frises de Coxeter sont des objets math\'ematiques qui poss\`edent des applications dans de nombreux domaines (voir \cite{Mo1}). En particulier, il existe des liens importants entre les frises de Coxeter et les matrices de la forme $\begin{pmatrix}
    x & -1 \\
    1    & 0 
   \end{pmatrix}$. Cela conduit \`a l'\'etude de l'\'equation matricielle suivante : \begin{equation}
\tag{$\tilde{E}$}
\label{q}
M_{n}(a_1,\ldots,a_n)=\begin{pmatrix}
   a_{n} & -1 \\
    1    & 0 
   \end{pmatrix}\cdots
   \begin{pmatrix}
   a_{1} & -1 \\
    1    & 0 
    \end{pmatrix}=-Id
\end{equation} o\`u $(a_1,\ldots,a_n) \in \bC^{n}$. En effet, les solutions de cette \'equation permettent de construire des frises de Coxeter, et, \`a partir d'une frise de Coxeter, on peut obtenir une solution de \eqref{q} (voir \cite{BR} et \cite{CH} proposition 2.4). Cette \'equation se g\'en\'eralise naturellement sous la forme ci-dessous : 
\begin{equation}
\label{p}
\tag{$E$}
M_{n}(a_1,\ldots,a_n)=\pm Id.
\end{equation} Les solutions de cette \'equation sont appel\'ees $\lambda$-quiddit\'es. Afin d'\'etudier ces solutions on introduit une notion d'irr\'eductibilit\'e d\'efinie gr\^ace \`a une op\'eration sur les $n$-uplets (voir \cite{C} et la section suivante). On cherche alors \`a conna\^itre les $\lambda$-quiddit\'es irr\'eductibles sur certaines sous-parties $A$ de $\bC$ c'est-\`a-dire \`a faire la liste de tous les $n$-uplets d'\'el\'ements de $A$ solutions irr\'eductibles de \eqref{p}. M.Cuntz a identifi\'e toutes les $\lambda$-quiddit\'es irr\'eductibles sur $\bZ$ (\cite{CH} Th\'eor\`eme 6.2) et sur $\bM$ (voir \cite{C} Th\'eor\`eme 3.1) et V.Ovsienko a r\'esolu l'\'equation \eqref{p} sur $\bN$ (voir \cite{O} Th\'eor\`eme 2). On peut alors se demander si on peut trouver toutes les solutions irr\'eductibles de \eqref{p} pour d'autres sous-ensembles de $\bC$ (voir \cite{C} probl\`eme ouvert 4.1). Notons que l'on peut \'egalement s'int\'eresser \`a cette \'equation sur d'autres anneaux que $\bC$, notamment sur les anneaux $\bZ/n\bZ$ (voir \cite{Ma,Mo2}).
\\
\\L'objectif ici est de r\'esoudre partiellement ce probl\`eme en d\'ecrivant l'ensemble des $\lambda$-quiddit\'es irr\'eductibles dans le cas de l'anneau $\bZ[\alpha]=\{P(\alpha),~P\in \bZ[X]\}$ o\`u $\alpha$ est un nombre complexe transcendant (c'est-\`a-dire qu'il n'existe pas de polyn\^ome non nul \`a coefficients entiers ayant $\alpha$ pour racine). Le r\'esultat principal est \'enonc\'e dans la section suivante et prouv\'e dans la section \ref{demo}, tandis que la section \ref{Rp} contient des r\'esultats utiles pour la preuve. Cette \'etude peut s'appliquer \`a un grand nombre de cas puisque l'ensemble des nombres r\'eels alg\'ebriques est d\'enombrable (voir \cite{Ca}). En particulier, celle-ci s'applique pour $\alpha=e$ (Th\'eor\`eme de Hermite, voir \cite{G} Théorème III.59) et pour $\alpha=\pi$ (Th\'eor\`eme de Lindemann, voir \cite{G} Théorème III.60). 

\section{R\'esultat principal}
\label{RP}

Soit $A$ un sous-anneau de $\bC$. On commence par d\'efinir formellement le concept de $\lambda$-quiddit\'e sur $A$.

\begin{defn}[\cite{C}, d\'efinition 2.2]
\label{21}
Soit $n \in \bN$. On dit que le $n$-uplet $(a_{1},\ldots,a_{n})$ d'\'el\'ements de $A$ est une $\lambda$-quiddit\'e sur $A$ de taille $n$ si $(a_{1},\ldots,a_{n})$ est une solution de \eqref{p}, c'est-\`a-dire si $M_{n}(a_{1},\ldots,a_{n})=\pm Id.$ S'il n'y a pas d'ambigu\"it\'e on parlera simplement de $\lambda$-quiddit\'e.

\end{defn}

Pour pouvoir \'etudier plus facilement les $\lambda$-quiddit\'es on a besoin des d\'efinitions suivantes :

\begin{defn}[\cite{C}, lemme 2.7]
\label{22}

Soient $(a_{1},\ldots,a_{n})$ un $n$-uplet d'\'el\'ements de $A$ et $(b_{1},\ldots,b_{m})$ un $m$-uplet d'\'el\'ements de $A$ (avec $(n,m) \in (\bN)^{2}$). On d\'efinit l'op\'eration suivante: \[(a_{1},\ldots,a_{n}) \oplus (b_{1},\ldots,b_{m}):= (a_{1}+b_{m},a_{2},\ldots,a_{n-1},a_{n}+b_{1},b_{2},\ldots,b_{m-1}).\] Le $(n+m-2)$-uplet ainsi obtenu est appel\'e la somme de $(a_{1},\ldots,a_{n})$ avec $(b_{1},\ldots,b_{m})$.

\end{defn} 

\begin{exam}On prend ici $A=\bR$. On a: 
\begin{itemize}
\item $(1,0,1) \oplus (1,2,1) = (2,0,2,2)$,
\item $(2,3,5) \oplus (1,0,7) = (9,3,6,0)$,
\item $(1,5,4,3) \oplus (2,4,4,6,2) = (3,5,4,5,4,4,6)$.
\end{itemize}

\end{exam}

Cette op\'eration est particuli\`erement utile puisqu'elle est \`a la base de la notion d'irr\'eductibilit\'e d\'efinie par M.Cuntz et rappel\'ee ci-dessous. Cependant, cette op\'eration n'est ni commutative ni associative (voir \cite{WZ} pour des contre-exemples).

\begin{defn}[\cite{C}, définition 2.5]
\label{23}

Soient $(a_{1},\ldots,a_{n})$ et $(b_{1},\ldots,b_{n})$ deux $n$-uplets d'\'el\'ements de $A$. On dit que $(a_{1},\ldots,a_{n}) \sim (b_{1},\ldots,b_{n})$ si $(b_{1},\ldots,b_{n})$ est obtenu par permutation circulaire de $(a_{1},\ldots,a_{n})$ ou de $(a_{n},\ldots,a_{1})$.

\end{defn}

$\sim$ est une relation d'\'equivalence sur l'ensemble des $n$-uplets d'\'el\'ements de $A$ (voir \cite{WZ}, lemme 1.7). De plus, si un $n$-uplet d'\'el\'ements de $A$ est une $\lambda$-quiddit\'e sur $A$ alors tout $n$-uplet d'\'el\'ements de $A$ qui lui est \'equivalent est aussi une $\lambda$-quiddit\'e sur $A$ (voir \cite{C} proposition 2.6). On peut d\'esormais d\'efinir la notion d'irr\'eductibilit\'e annonc\'ee.

\begin{defn}[\cite{C}, d\'efinition 2.9]
\label{24}

Une $\lambda$-quiddit\'e $(c_{1},\ldots,c_{n})$ sur $A$ avec $n \geq 3$ est dite r\'eductible s'il existe deux $\lambda$-quiddit\'es $(a_{1},\ldots,a_{m})$ et $(b_{1},\ldots,b_{l})$ sur $A$ telles que \begin{itemize}
\item $(c_{1},\ldots,c_{n}) \sim (a_{1},\ldots,a_{m}) \oplus (b_{1},\ldots,b_{l})$,
\item $m \geq 3$ et $l \geq 3$.
\end{itemize}
Une $\lambda$-quiddit\'e est dite irr\'eductible si elle n'est pas r\'eductible.
\end{defn}

\begin{rem} 

{\rm $(0,0)$ est toujours solution de l'\'equation \eqref{p} mais elle n'est pas consid\'er\'ee comme \'etant une solution irr\'eductible de celle-ci.}

\end{rem}

On s'int\'eresse ici aux $\lambda$-quiddit\'es irr\'eductibles sur l'anneau $\bZ[\alpha]$ dans le cas o\`u $\alpha$ est un nombre complexe transcendant. On a le r\'esultat suivant :

\begin{thm}
\label{25}

Soit $\alpha$ un nombre complexe transcendant. Les $\lambda$-quiddit\'es irr\'eductibles sur l'anneau $\bZ[\alpha]$ sont : \[\{(1,1,1), (-1,-1,-1), (0,P(\alpha),0,-P(\alpha)), (P(\alpha),0,-P(\alpha),0); P \in \bZ[X]-\{\pm 1\} \}.\]

\end{thm}

Ce r\'esultat est prouv\'e dans la section \ref{demo}.

\section{R\'esultats pr\'eliminaires}
\label{Rp}

On donne dans cette section des r\'esultats qui nous seront utiles pour d\'emontrer le th\'eor\`eme \ref{25}.

\subsection{G\'en\'eralit\'es sur les $\lambda$-quiddit\'es}
\label{G}

Soit $A$ un sous-anneau de $\bC$. On commence par chercher les solutions de \eqref{p} pour les petites valeurs de $n$.

\begin{pro}
\label{31}
\begin{itemize}
\item $(0,0)$ est la seule solution de \eqref{p} de taille 2.
\item $(1,1,1)$ et $(-1,-1,-1)$ sont les seules solutions de \eqref{p} de taille 3.
\item Les solutions de \eqref{p} pour $n=4$ sont les 4-uplets suivants $(-a,b,a,-b)$ avec $ab=0$ et $(a,b,a,b)$ avec $ab=2$.
\end{itemize}

\end{pro}

\begin{proof}

La preuve se r\'eduit \`a des calculs matriciels (voir par exemple \cite{CH} exemple 2.7).

\end{proof}

\begin{pro}
\label{32}

Soit $(b_{1},\ldots,b_{m})$ une $\lambda$-quiddit\'e. Soit $(a_{1},\ldots,a_{n}) \in A^{n}$. La somme $(a_{1},\ldots,a_{n}) \oplus (b_{1},\ldots,b_{m})$ est une $\lambda$-quiddit\'e si et seulement si $(a_{1},\ldots,a_{n})$~est une $\lambda$-quiddit\'e.

\end{pro}

\begin{proof}

La preuve est essentiellement calculatoire. Pour le d\'etail des calculs, on peut consulter la preuve du lemme 2.7 de \cite{C} et la preuve du lemme 1.9 de \cite{WZ}.

\end{proof}

Une des cons\'equences imm\'ediates de ce r\'esultat est qu'il n'est pas n\'ecessaire que $(a_{1},\ldots,a_{m})$ soit une $\lambda$-quiddit\'e dans la d\'efinition \ref{24}. Ceci nous permet d'avoir le r\'esultat suivant :

\begin{pro}
\label{33}

i) Les $\lambda$-quiddit\'es de taille 3 sont irr\'eductibles.
\\ii)Une $\lambda$-quiddit\'e de taille 4 r\'eductible contient $1$ ou $-1$.
\\iii) Si $n \geq 4$ alors une $\lambda$-quiddit\'e de tille 4 contenant $1$ ou $-1$ est r\'eductible.
\\ iv) Une $\lambda$-quiddit\'e de taille supérieure à 5 contenant $0$ est r\'eductible.

\end{pro}

\begin{proof}

i)Si un 3-uplet est somme d'un $m$-uplet avec un $l$-uplet alors on a $3=m+l-2$ et donc $m+l=5$. Ceci implique $m \leq 2$ ou $l \leq 2$. Donc, les $\lambda$-quiddit\'es de taille 3 sont irr\'eductibles.
\\
\\ii)Soit $(a_{1},a_{2},a_{3},a_{4})$ une $\lambda$-quiddit\'e. 
\\
\\Si $(a_{1},a_{2},a_{3},a_{4})$ est r\'eductible alors $(a_{1},a_{2},a_{3},a_{4})$ est \'equivalent \`a la somme d'un $m$-uplet solution de \eqref{p} avec un $l$-uplet solution de \eqref{p} avec $m, l \geq 3$. On a $m+l-2=4$ donc $m+l=6$ et comme $m, l \geq 3$ on a n\'ecessairement $m=l=3$. Comme les $\lambda$-quiddit\'es de taille 3 contiennent $1$ ou $-1$ (par i)), une $\lambda$-quiddit\'e r\'eductible de taille 4 contient $1$ ou $-1$.
\\
\\iii)Soit $(a_{1},\ldots,a_{n})$ (avec $n \geq 4$) une solution de \eqref{p}. S'il existe $\epsilon$ dans $\{-1, 1\}$ et s'il existe $i$ dans $[\![1;n]\!]$ tels que $a_{i}=\epsilon$ alors on a :\[(a_{i+1},\ldots,a_{n},a_{1},\ldots,a_{i})=(a_{i+1}-\epsilon,\ldots,a_{n},a_{1},\ldots,a_{i-1}-\epsilon) \oplus (\epsilon,\epsilon,\epsilon).\]
Comme $(\epsilon,\epsilon,\epsilon)$ est solution de \eqref{p} (voir proposition \ref{31}), $(a_{1},\ldots,a_{n})$ est une $\lambda$-quiddit\'e r\'eductible.
\\
\\iv) Soit $(a_{1},\ldots,a_{n})$ (avec $n \geq 5$) une solution de \eqref{p}. S'il existe $i$ dans $[\![1;n]\!]$ tel que $a_{i}=0$ alors on a :

\begin{eqnarray*}
(a_{i+2},\ldots,a_{n},a_{1},\ldots,a_{i},a_{i+1}) &=& (a_{i+2},\ldots,a_{n},a_{1},\ldots,a_{i-1}+a_{i+1}) \\
                                                  &\oplus & (-a_{i+1},0,a_{i+1},0).\\
\end{eqnarray*}

Comme $(-a_{i+1},0,a_{i+1},0)$ est solution de \eqref{p} (voir proposition \ref{31}), on a que $(a_{1},\ldots,a_{n})$ est une $\lambda$-quiddit\'e r\'eductible.

\end{proof}

On dispose \'egalement du r\'esultat suivant qui permet d'avoir une majoration des $\left|a_{i}\right|$ :

\begin{thm}[Cuntz-Holm, \cite{CH} corollaire 3.3]
\label{34}

Soit $(a_{1},\ldots,a_{n}) \in \bC^{n}$ une $\lambda$-quiddit\'e. $\exists (i,j) \in [\![1;n]\!]^{2}$, $i \neq j$, tels que $\left|a_{i}\right| < 2$ et $\left|a_{j}\right| < 2$.

\end{thm}

Ce r\'esultat permet de r\'esoudre la question de la recherche des $\lambda$-quiddit\'es irr\'eductibles pour de nombreux sous-anneaux $A$ de $\bC$. Par exemple, pour l'anneau $A=\bZ[2i]=\{a+2bi,~(a,b) \in \bZ\}$.
\\
\\En effet, si $z=a+2bi \in \bZ[2i]$ alors $\left|z\right|=\sqrt{a^{2}+4b^{2}} \geq 2\left|b\right|$. Ainsi, $\left|z\right| < 2$ si et seulement si $z \in \{-1, 0, 1\}$. Donc, les $\lambda$-quiddit\'es sur $\bZ[2i]$ de taille sup\'erieure \`a 5 sont r\'eductibles par la proposition \ref{33}. Par ailleurs, on montre facilement que les solutions de taille 4 de la forme $(a,b,a,b)$, avec la condition $ab=2$, contiennent toujours $\pm 1$ et donc sont r\'eductibles par la proposition \ref{33}. Donc, Les $\lambda$-quiddit\'es irr\'eductibles sur $\bZ[2i]$ sont les \'el\'ements suivants : $(1,1,1)$, $(-1,-1,-1)$, $(-z,0,z,0)$ et $(0,-z,0,z)$ (avec $z \in \bZ[2i]-\{\pm  1\}$). En particulier, on remarque que le cas de $\bZ[2i]$ est beaucoup plus simple que celui de de $\bZ[i]$ (voir \cite{C} proposition 3.3).
\\
\\Bien que l'on n'ait pas besoin de cela dans la suite on peut remarquer que la constante $2$ du th\'eor\`eme pr\'ec\'edent est optimale. En effet, on a le r\'esultat suivant:

\begin{pro}

Soit $\epsilon \in ]0,2]$. Il existe une solution de \eqref{p} dont toutes les composantes sont de module sup\'erieur \`a $2-\epsilon$.

\end{pro}

\begin{proof}

Posons $\forall n \geq 2$, $u_{n}=2cos(\frac{\pi}{n})$ et $A_{n}=\begin{pmatrix}
   u_{n} & -1 \\[4pt]
    1    & 0 
   \end{pmatrix}$.
\\
\\Soit $\epsilon \in ]0,2]$. $\lim\limits_{n \rightarrow +\infty} u_{n}=2$, donc $\exists n \in \bN$, $n \geq 2$, tel que $u_{n} > 2-\epsilon$. Le polyn\^ome caract\'eristique de $A_{n}$ est :\[\chi_{A_{n}}(X)=det(A_{n}-XId)=-X(u_{n}-X)+1=X^{2}-u_{n}X+1.\] Le discriminant de ce polyn\^ome est $\Delta=u_{n}^{2}-4~<~0$. $\chi_{A_{n}}$ a deux racines complexes conjugu\'ees $x_{1}$ et $x_{2}$ avec : \[x_{1}=\frac{u_{n}+i\sqrt{-u_{n}^{2}+4}}{2}~{\rm et}~x_{2}=\frac{u_{n}-i\sqrt{-u_{n}^{2}+4}}{2}.\]

Or, on a :
\begin{eqnarray*}
x_{1} &=& \frac{u_{n}+i\sqrt{-u_{n}^{2}+4}}{2} \\
      &=& cos(\frac{\pi}{n})+i\sqrt{-cos(\frac{\pi}{n})^{2}+1} \\
			&=& cos(\frac{\pi}{n})+i\sqrt{sin(\frac{\pi}{n})^{2}} \\
			&=& cos(\frac{\pi}{n})+i~sin(\frac{\pi}{n})~{\rm (car }~sin(\frac{\pi}{n}) \geq 0) \\
			&=& e^{\frac{i\pi}{n}}. \\
\end{eqnarray*}

Donc, $A_{n}$ est diagonalisable et les valeurs propres de $A_{n}$ sont $e^{\frac{i\pi}{n}}$ et $e^{\frac{-i\pi}{n}}$. Ainsi, $(A_{n})^{n}=-Id$. Donc, le $n$-uplet form\'e uniquement de $2cos(\frac{\pi}{n})$ est solution de \eqref{p} et $2cos(\frac{\pi}{n}) > 2-\epsilon$.

\end{proof}

\subsection{Continuants}
\label{C}

On aura \'egalement besoin dans la suite du r\'esultat suivant sur l'expression de la matrice $M_{n}(a_{1},\ldots,a_{n})$ en terme de d\'eterminant. On pose $K_{-1}=0$, $K_{0}=1$ et on note pour $i \geq 1$ \[K_i(a_{1},\ldots,a_{i})=
\left|
\begin{array}{cccccc}
a_1&1&&&\\[4pt]
1&a_{2}&1&&\\[4pt]
&\ddots&\ddots&\!\!\ddots&\\[4pt]
&&1&a_{i-1}&\!\!\!\!\!1\\[4pt]
&&&\!\!\!\!\!1&\!\!\!\!a_{i}
\end{array}
\right|.\] $K_{i}(a_{1},\ldots,a_{i})$ est le continuant de $a_{1},\ldots,a_{i}$. On dispose de l'\'egalit\'e suivante (voir \cite{CO,MO}) : \[M_{n}(a_{1},\ldots,a_{n})=\begin{pmatrix}
    K_{n}(a_{1},\ldots,a_{n}) & -K_{n-1}(a_{2},\ldots,a_{n}) \\
    K_{n-1}(a_{1},\ldots,a_{n-1})  & -K_{n-2}(a_{2},\ldots,a_{n-1}) 
   \end{pmatrix}.\]
	
\section{D\'emonstration du th\'eor\`eme}
\label{demo}

On peut maintenant d\'emontrer le th\'eor\`eme \ref{25}. Dans la suite, si $P$ est un polyn\^ome, on notera $deg(P)$ le degr\'e de $P$ et si $x$ est un r\'eel on notera $E[x]$ sa partie enti\`ere.

\begin{proof}

Soit $\alpha$ un nombre complexe transcendant. Les \'el\'ements de $\bZ[\alpha]$ sont de la forme $P(\alpha)$ avec $P$ un polyn\^ome \`a coefficients entiers.
\\
\\Commen\c{c}ons par trouver les $\lambda$-quiddit\'es irr\'eductibles sur $\bZ[\alpha]$ de taille 3 et 4. Par la proposition \ref{31}, $(1,1,1)$ et $(-1,-1,-1)$ sont les seules $\lambda$-quiddit\'es de taille 3 sur $\bZ[\alpha]$, et, par la proposition \ref{33}, elles sont irr\'eductibles. Par la proposition \ref{31}, les $\lambda$-quiddit\'es de taille 4 sur $\bZ[\alpha]$ sont les 4-uplets de la forme $(P(\alpha),0,-P(\alpha),0)$, $(0,P(\alpha),0,-P(\alpha))$ (avec $P \in \bZ[X]$) qui sont \'equivalents et $(Q(\alpha),R(\alpha),Q(\alpha),R(\alpha))$ avec $R(\alpha)Q(\alpha)=2$. 
\\
\\Soit $(Q,R) \in \bZ[X]$ avec $R(\alpha)Q(\alpha)=2$. Posons $T(X)=R(X)Q(X)$. Si $deg(T) \geq 1$ alors $T(X)-2$ est un polyn\^ome non constant \`a coefficients entiers annulant $\alpha$. Ceci est absurde car $\alpha$ est transcendant. Ainsi, $T$ est constant et donc $Q$ et $R$ sont constants. En particulier, un de ces deux polyn\^omes vaut n\'ecessairement 1 ou -1 et donc la $\lambda$-quiddit\'e est r\'eductible par la proposition \ref{33}.
\\
\\Une solution \'equivalente \`a $(P(\alpha),0,-P(\alpha),0)$ est r\'eductible si et seulement si $P(\alpha)= \pm 1$ (voir proposition \ref{33}). Or, si $P(\alpha)=\pm 1$ alors $P$ est constant \'egal \`a $\pm 1$ (car sinon $P(X)-1$ est un polyn\^ome non constant \`a coefficients entiers annulant $\alpha$ ce qui est impossible).
\\
\\On va maintenant montrer que les $\lambda$-quiddit\'es irr\'eductibles sur $\bZ[\alpha]$ sont de taille inf\'erieure \`a 4. Soient $n \in \bN$ et $(P_{1}(\alpha),\ldots,P_{n}(\alpha))$ une $\lambda$-quiddit\'e sur $\bZ[\alpha]$. On commence par montrer qu'un des $P_{i}$ est n\'ecessairement constant.
\\
\\$(P_{1}(\alpha),\ldots,P_{n}(\alpha))$ est une $\lambda$-quiddit\'e donc il existe $\epsilon$ dans $\{-1, 1 \}$ tel que \[M_{n}(P_{1}(\alpha),\ldots,P_{n}(\alpha))=\epsilon Id.\] Par les r\'esultats de la section pr\'ec\'edente on a :\[\begin{pmatrix}
    K_{n}(P_{1}(\alpha),\ldots,P_{n}(\alpha)) & -K_{n-1}(P_{2}(\alpha),\ldots,P_{n}(\alpha)) \\
    K_{n-1}(P_{1}(\alpha),\ldots,P_{n-1}(\alpha))  & -K_{n-2}(P_{2}(\alpha),\ldots,P_{n-1}(\alpha)) 
   \end{pmatrix}=\epsilon Id.\]  Posons $Q(X)=K_{n}(P_{1}(X),\ldots,P_{n}(X))$. 
\\
\\$Q(X) \in \bZ[X]$ car $K_{n}(x_{1},\ldots,x_{n}) \in \bZ[x_{1},\ldots,x_{n}]$ et les $P_{i}$ sont \`a coefficients entiers. $Q$ est un polyn\^ome constant \'egal \`a $\epsilon$. En effet, si $deg(Q(X)) \geq 1$ alors $R(X)=Q(X)-\epsilon$ est un polyn\^ome non constant \`a coefficients entiers annulant $\alpha$. Ceci est absurde puisque $\alpha$ est transcendant.
\\
\\Ainsi, $K_{n}(P_{1}(X),\ldots,P_{n}(X))=\epsilon$. De m\^eme, $K_{n-2}(P_{2}(X),\ldots,P_{n-1}(X))=-\epsilon$ et $K_{n-1}(P_{2}(X),\ldots,P_{n}(X))=K_{n-1}(P_{1}(X),\ldots,P_{n-1}(X))=0$. On en d\'eduit que si $a\in \bZ$ alors $(P_{1}(a),\ldots,P_{n}(a))$ est une $\lambda$-quiddit\'e sur $\bZ$. 
\\
\\Supposons par l'absurde que pour tout $i$ appartenant \`a $[\![1;n]\!]$ $deg(P_{i}) \geq 1$.
\\
\\ $\forall i \in [\![1;n]\!]$, $P_{i}$ n'est pas constant donc $\lim\limits_{x\rightarrow +\infty} \left| P_{i}(x)\right| = +\infty$.
\\
\\Donc, $\forall i \in [\![1;n]\!]$, $\exists M_{i} \in \bR$ tel que $\forall x \geq M_{i}$, $\left| P_{i}(x)\right| \geq 3$. Donc, si on note $M={\rm max}(M_{i}, i \in [\![1;n]\!])$ on a pour tout r\'eel $x$ sup\'erieur \`a $M$ et pour tout $i$ dans $[\![1;n]\!]$, $\left| P_{i}(x)\right| \geq 3$.
\\
\\Posons $a=E[M]+1>M$. $(P_{1}(a),\ldots,P_{n}(a))$ est une $\lambda$-quiddit\'e sur $\bZ$ et pour tout $i$ dans $[\![1;n]\!]$, $\left| P_{i}(a)\right| \geq 3$. Ceci est absurde par le th\'eor\`eme \ref{34}. 
\\
\\Donc, un des $P_{i}$ est constant. On souhaite maintenant montrer qu'un des polyn\^omes constants vaut 0, 1 ou -1. Notons $i_{1},\ldots,i_{r}$ les \'el\'ements de $[\![1;n]\!]$ pour lesquels les $P_{i}$ sont constants.
\\
\\Puisque $(P_{1}(a),\ldots,P_{n}(a))$ est une $\lambda$-quiddit\'e sur $\bZ$, $\exists j_{0} \in [\![1;n]\!]$ tel que $\left| P_{j_{0}}(a)\right|<2$ par le th\'eor\`eme \ref{34}. Comme $P_{j_{0}}(a) \in \bZ$, on a $P_{j_{0}}(a) \in \{-1, 0, 1\}$. Par ce qui pr\'ec\`ede, $j_{0} \in \{i_{1},\ldots,i_{r}\}$ et donc un des $P_{i}$ est constant \'egal \`a 0, 1 ou -1.
\\
\\Donc, les $\lambda$-quiddit\'es sur $\bZ[\alpha]$ de taille sup\'erieure \`a 5 sont r\'eductibles (voir proposition \ref{33}). Ceci conclut la preuve. 

\end{proof}

\end{document}